\documentclass[11pt]{article}
\usepackage{verbatim} 
\usepackage{amsmath}
\usepackage{amssymb}
\usepackage{amsfonts}
\usepackage{amsthm}
\usepackage{cite}
\usepackage{graphicx} 
\usepackage{mathrsfs}
\usepackage{color}
\usepackage{enumitem}
\usepackage{latexsym}
\usepackage{enumerate}

\usepackage{pgf,tikz,pgfplots}  
\usetikzlibrary{arrows}

\newtheorem{theorem}{Theorem}
\newtheorem{corollary}[theorem]{Corollary}
\newtheorem{lemma}[theorem]{Lemma}
\newtheorem{proposition}[theorem]{Proposition}

\theoremstyle{definition}

\newtheorem{definition}[theorem]{Definition}
\newtheorem{question}[theorem]{Question}

\newcommand{\R}{\mathbb{R}}
\newcommand{\N}{\mathbb{N}}

\begin{document}   
     
\makeatletter
\@namedef{subjclassname@2010}{
  \textup{2010} Mathematics Subject Classification 03E15, 54H05 and 54F15}
\makeatother

\title{There is no compact metrizable space containing all continua as unique components}

\author{
Benjamin Vejnar\footnote{This work has been supported by Charles University Research Centre program No.UNCE/SCI/022.
}
\\ Department of Mathematical Analysis
\\ Faculty of Mathematics and Physics, Charles University, Czechia\\
\\ Email: vejnar@karlin.mff.cuni.cz
}

\maketitle

\begin{abstract}
We answer a question of Piotr Minc by proving that there is no compact metrizable space whose set of components contains a unique topological copy of every metrizable compactification of a ray (i.e. a half-open interval) with an arc (i.e. closed bounded interval) as the remainder.
To this end we use the concept of Borel reductions coming from Invariant descriptive set theory. It follows as a corollary that there is no compact metrizable space such that every continuum is homeomorphic to exactly one component of this space.
\end{abstract}

\section{Introduction}
By a continuum we mean a compact connected metrizable space. It can be easily observed that there is a compact metrizable space $X$ such that every continuum is homeomorphic to some component of $X$. Indeed, let us consider the Cantor set $C$, the Hilbert cube $Q$, the hyperspace of all subcontinua $\mathcal C(Q)$ with the Vietoris topology (which is known to be a compact metrizable space), and any continuous surjection $f\colon C\to \mathcal C(Q)$ (which exists by elementary properties of the Cantor set), and let 
\[X=\{(c, x)\in C\times Q\colon x\in f(c)\}\]
with the topology inhereted from the product space $C\times Q$.
To what extent this result can be generalised in a positive direction was the main scope of \cite{BartosBMPV}. One of the aims of this paper is to prove a negative counterpart, namely that there is no such a space $X$ in which all the components are pairwise non-homeomorphic.

For a continuum $K$ let us call a space $S$ a \emph{spiral over $K$} if $S$ is homeomorphic to a metrizable compactification of $[0,\infty)$ with $K$ as the remainder.
It is an old result of Waraszkiewicz that there are uncountably many spirals over a circle such that no one can be mapped onto any other by a continuous mapping \cite{Waraszkiewicz} (see \cite{PyrihVejnar} for a simple proof).
Awartani constructed an uncountable collection of compactifications of a ray with the same properties \cite{Awartani}.
Barto\v s, Marci\v na, Pyrih and the author proved that for every non-degenerate Peano continuum $K$ there is a family of size continuum of spirals over $K$ such that no one can be mapped continuously onto any other \cite{BartosMarcinaPyrihVejnar}.
On the other hand Illanes, Minc and Sturm proved that for every pair $Y, Z$ of spirals over the pseudo-arc, $Y$ can be mapped continuously onto $Z$ \cite{IllanesMincSturm}.
Recently, Minc proved that for every non-degenerate continuum $K$ there is a perfect set of spirals over $K$ such that no one is homeomorphic to any other \cite{Minc}.

Minc asked whether for a non-degenerate continuum $K$ there is compact metrizable space $X$ such that
(1) every spiral over $K$ is homeomorphic to a \emph{unique} component of $K$ and at the same time
(2) every component of $X$ is homeomorphic to a spiral over $K$
\cite[Question 2]{Minc}.
In the next section a negative answer to the question is presented. We prove that there is no compact metrizable space $X$ satisfying (1) with $K=[0,1]$. The proof works also with another choices of the space $K$, e.g. if $K$ is supposed to be a circle. 

\section{Main results}

A \emph{Polish space} is a completely metrizable separable space.
Let us denote by $\mathcal K(X)$ the \emph{hyperspace} of all compact subsets of a space $X$ with the \emph{Vietoris topology}. This is well known to be a Polish space if $X$ is Polish.
Recall that a continuous surjective mapping is called \emph{perfect} if it is closed and preimages of points are compact.
The following definition is suitable for our purposes.

\begin{definition}\label{compactifiable}
A collection $\mathcal C$ of metrizable continua is called \emph{uniquely compactifiable} if there exists a compact metrizable space $X$ such that for every $C\in\mathcal C$ there exists a unique component $D$ of $X$ which is homeomorphic to $C$.
\end{definition}

Let us focus on the fact that in the definition above we do not require that the components of $X$ are homeomorphic to elements of $\mathcal C$.

Let us recall some basic definitions from Invariant descriptive set theory. See \cite{Gao} for more details in this field
or \cite{Foreman} for a nice and short introduction to the theory of Borel reductions.
Let $X$, $Y$ be Polish spaces and $E, F$ equivalence relations on $X, Y$ respectively. We say that $E$ is \emph{Borel reducible} to $F$ if there is a Borel measurable mapping $\varphi\colon X\to Y$ such that $(a,b)\in E$ if and only if $(\varphi(a),\varphi(b))\in F$ for every $a,b\in X$.
An equivalence relation is called \emph{smooth} if it is Borel reducible to the equality equivalence relation on a Polish space.

We can relax the notion of unique compactifiability in the sense of the following definition.

\begin{definition}\label{polishable}
A collection $\mathcal C$ of compact metrizable spaces is called \emph{uniquely Polishable} if there exist Polish spaces $X, Y$ and a perfect mapping $\pi\colon X\to Y$ such that for every $C\in\mathcal C$ there exists unique $y\in Y$ for which $\pi^{-1}(y)$ is homeomorphic to $C$.
\end{definition}

\begin{lemma}\label{comppoli}
If $\mathcal C$ is a collection of continua which is uniquely compactifiable then $\mathcal C$ is uniquely Polishable.
\end{lemma}

\begin{proof}
Since $\mathcal C$ is compactifiable, there is a witnessing compact space $X$ as in Definition \ref{compactifiable}.
Let us denote by $Y$ the quotient space of $X$ where all the components shrink to points. It is known that $Y$ is a compact metrizable space.
Let us denote as $\pi\colon X\to Y$ the corresponding quotient mapping.
Clearly $\pi$ is a perfect mapping and hence $\mathcal C$ is uniquely Polishable.
\end{proof}

We denote by $Q$ the Hilbert cube and by $H$ the homeomorphism equivalence relation on $\mathcal K(Q)$. It means that $(K,L)\in H$ if and only if $K, L\in\mathcal K(Q)$ and $K$ is homeomorphic to $L$. 

The following theorem contains one of the two main ingredients towards the solution of Minc' question.

\begin{theorem}\label{uniquePolIsSmooth}
Suppose that $\mathcal C\subseteq\mathcal K(Q)$ is uniquely Polishable. Then for every Polish space $T\subseteq \mathcal C$ the equivalence relation $H\restriction T$ is smooth.
\end{theorem}

\begin{proof}
Since $\mathcal C$ is uniquely Polishable, there are Polish spaces $X, Y$ and a continuous mapping $\pi\colon X\to Y$ as in Definition \ref{polishable}. Moreover we can suppose that $X\subseteq Q$, whence $\mathcal K(X)\subseteq\mathcal K(Q)$. Let us denote by $\rho\colon Y\to \mathcal K(X)$ the mapping defined as $\rho(y)=\pi^{-1}(y)$.
Let $R=\{\pi^{-1}(y)\colon y\in Y\}=\rho(Y)$.
Since $\pi$ is closed it follows by \cite[Proposition 3.16]{BartosBMPV}, 
that the set $R$ is a $G_\delta$-set in the hyperspace $\mathcal K(X)$.
Since moreover $\mathcal K(X)$ is a Polish space, it follows that $R$ is a Polish space and also a $G_\delta$-set in $\mathcal K(Q)$ by \cite[Theorem 3.11]{Kechris}.
The equivalence relation $H$ is an analytic set in $\mathcal K(Q)^2$ \cite[Proposition 14.4.3]{Gao} and thus the binary relation $H\cap (T\times R)$ is analytic since $T\times R$ is $G_\delta$ in $\mathcal K(Q)^2$.
Moreover all the vertical sections of $H\cap (T\times R)$ are singletons.
Thus $H\cap (T\times R)$ is a graph of a mapping $\varphi\colon T\to R$. 
Since by \cite[Theorem 14.12]{Kechris} every mapping between Polish spaces with an analytic graph is Borel measurable, it follows that the mapping $\varphi$ is Borel measurable.
Clearly, if $A,B\in T$ then $A$ is homeomorphic to $B$ if and only if $\varphi(A)=\varphi(B)$.
Thus $H\restriction T$ is Borel reducible to the equality on the Polish space $R$,
hence $H\restriction T$ is smooth.
\end{proof}

Let us recall that \emph{a spiral over a continuum $K$} is any space homeomorphic to a metrizable compactification of $(0,1]$ with $K$ as the remainder.
Let us denote by $\mathcal S(K)$ the collection of all spirals over $K$ which are contained in $\mathcal K(Q)$. A mapping $f\colon (0,1]\to K$ is called \emph{compactifying} if $\bigcap_{\varepsilon>0}\overline{f(0,\varepsilon)}=K$. Note that in this case $S_f:=\text{graph}(f)\cup(\{0\}\times K)$ forms a compactification of $(0,1]$ with remainder $\{0\}\times K$.
Let us denote by $\mathcal H(X)$ the homeomorphism group of a topological space $X$.

The following lemma is a modification of both \cite[Theorem 2]{PyrihVejnar} and \cite[Proposition 3.3]{BartosMarcinaPyrihVejnar}. Its brief proof is involved for the sake of completeness.

\begin{lemma}\label{homeo}
If $K$ is a non-degenerate continuum and $e, f\colon(0,1]\to K$ are compactifying
then
$S_{e}$ is homeomorphic to $S_{f}$ if and only if there exist $\alpha\in\mathcal H(K)$ and $\beta\in\mathcal H((0,1])$ such that 
\[\lim_{t\to 0^+} |\alpha (e(t))-f(\beta(t))|=0.\]
\end{lemma}

\begin{proof}
The graph of $g$ is clearly the only dense arc-component of $S_g$.
Hence if $h\colon S_e\to S_f$ is a homeomorphism it follows that 
$h(\text{graph}(e))=\text{graph}(f)$ and thus $h(\{0\}\times K)=\{0\}\times K$.
Let $\alpha(x)=\pi_2(h(0,x))$ and
let $\beta(x)=\pi_1(h(x,e(x)))$.
It is straightforward to verify that $\alpha$ and $\beta$ are homeomorphisms satisfying the limit condition.

For the converse implication one can set $h(0,x)=(0,\alpha(x))$, $x\in K$, and $h(t, e(t))=(\beta(t), f(\beta(t)))$, $t\in(0,1]$. It is a routine to verify that $h\colon S_e\to S_f$ is a homeomorphism.
\end{proof}

Let us denote $2=\{0,1\}$ and let $E_0$ be the equivalence relation of eventual equality of sequence in $2^\N$, that is $(x,y)\in E_0$ if and only if $x,y\in 2^\N$ and there exists $m\in\N$ such that $x_n=y_n$ for $n\geq m$.

Let $f\colon J\to\mathbb R$ be a continuous mapping of an interval. We say that a point $p\in J$ is \emph{a peak point for $f$ of height $h>0$} if there is an interval $[a,b]\subseteq J$ containing $p$ in the interior with the property that $f$ is increasing on $[a,p]$, $f$ is decreasing of $[p,b]$ and $\max\{f(a),f(b)\}=f(p)-h$. We say that $P=\{p_1<\dots<p_n\}$ are \emph{consecutive peak points for $f$} if each $p_i$, $i\leq n$, is a peak point for $f$ and there is no peak point for $f$ in $[p_1,p_n]$ except those in the set $P$.
The following proposition contains the second ingredient to the solution of Minc' question.

\begin{proposition}\label{Enula}
The equivalence relation $E_0$ is Borel reducible to $H\restriction T$ for some Polish space $T\subseteq\mathcal S([0,1])$.
\end{proposition}

\begin{proof}
In what follows, we describe a continuous mapping $\psi\colon 2^\N\to \mathcal S([0,1])$ such that $(a,b)\in E_0$ if and only if $\psi(a)$ is homeomorphic to $\psi(b)$. Consequently we can define $T=\psi(2^\N)$ which is clearly a compact set and thus a Polish space.

Let $(c_n)$ be a decreasing sequence in $(0,1]$ with $c_1=1$ and whose limit is 0.
Let $u_n \in (c_{4n-2},c_{4n-3})$, $v_n\in (c_{4n-1},c_{4n-2})$ and $w_n\in (c_{4n},c_{4n-1})$ are chosen arbitrarily. 
Let $A_n, B_n \subseteq (c_{4n+1},c_{4n})$ be disjoint finite sets with $|A_n|=n$, $|B_n|=n-1$ such that between any two distinct elements of $A_n$ there is a point of $B_n$.

\begin{tikzpicture}[line cap=round,line join=round,>=triangle 45,x=1.0cm,y=1.0cm]
\clip(-0.8,-2.0) rectangle (11.0,4.3);
\draw [line width=0.4pt] (0.,3.)-- (0.,0.);
\draw [line width=0.4pt] (8.,0.)-- (7.5,1.);
\draw [line width=0.4pt] (7.,0.)-- (7.5,1.);
\draw [line width=0.4pt] (7.,0.)-- (6.5,2.);
\draw [line width=0.4pt] (6.5,2.)-- (6.,0.);
\draw [line width=0.4pt,dash pattern=on 2pt off 2pt] (6.,0.)-- (5.5,1.);
\draw [line width=0.4pt,dash pattern=on 2pt off 2pt] (5.5,1.)-- (5.,0.);
\draw [line width=0.4pt,dash pattern=on 2pt off 2pt] (5.,0.)-- (5.5,2.);
\draw [line width=0.4pt,dash pattern=on 2pt off 2pt] (5.5,2.)-- (6.,0.);
\draw [line width=0.4pt] (5.,0.)-- (4.5,3.);
\draw [line width=0.4pt] (4.5,3.)-- (4.,0.);
\draw [line width=0.4pt] (4.,0.)-- (3.5,3.);
\draw [line width=0.4pt] (3.5,3.)-- (3.,0.);
\draw [line width=0.4pt] (3.,0.)-- (2.5,3.);
\draw [line width=0.4pt] (2.5,3.)-- (2.,0.);
\draw [line width=0.4pt,dotted] (2.,0.)-- (1.8,0.6);
\draw [line width=0.4pt,dotted] (8.,0.)-- (8.2,0.4);
\begin{scriptsize}
\draw [fill=black] (0.,3.) circle (0.5pt);
\draw[color=black] (-0.19433545228051693,2.9417540418391974) node {3};
\draw [fill=black] (0.,0.) circle (0.5pt);
\draw[color=black] (-0.15875481464981897,-0.03) node {0};
\draw [fill=black] (10.,0.) circle (0.5pt);
\draw[color=black] (9.995069649210613,-0.55) node {$c_1=1$};
\draw [fill=black] (8.,0.) circle (0.5pt);
\draw[color=black] (7.975868463668505,-0.55) node {$c_{4n-3}$};
\draw [fill=black] (7.,0.) circle (0.5pt);
\draw[color=black] (6.926239653562915,-0.55) node {$c_{4n-2}$};
\draw [fill=black] (6.,0.) circle (0.5pt);
\draw[color=black] (5.938876959311046,-0.55) node {$c_{4n-1}$};
\draw [fill=black] (5.,0.) circle (0.5pt);
\draw[color=black] (4.987094902689876,-0.55) node {$c_{4n}$};
\draw [fill=black] (7.5,1.) circle (0.5pt);
\draw [fill=black] (6.5,2.) circle (0.5pt);
\draw [fill=black] (5.5,1.) circle (0.5pt);
\draw [fill=black] (5.5,2.) circle (0.5pt);
\draw[color=black] (5.725393133526858,2.537024288790008) node {$1+a_n$};
\draw [fill=black] (4.5,3.) circle (0.5pt);
\draw [fill=black] (4.,0.) circle (0.5pt);
\draw [fill=black] (3.5,3.) circle (0.5pt);
\draw[color=black] (3.4571274845698623,3.4888063454111786) node {$A_n$};
\draw [fill=black] (3.,0.) circle (0.5pt);
\draw[color=black] (3.466022643977537,-0.23) node {$B_n$};
\draw [fill=black] (2.5,3.) circle (0.5pt);
\draw [fill=black] (2.,0.) circle (0.5pt);
\draw[color=black] (1.9983213417112453,-0.55) node {$c_{4n+1}$};
\draw [fill=black] (0.5,1.5) circle (0.5pt);
\draw [fill=black] (1.,1.5) circle (0.5pt);
\draw [fill=black] (1.5,1.5) circle (0.5pt);
\draw [fill=black] (8.5,1.5) circle (0.5pt);
\draw [fill=black] (9.,1.5) circle (0.5pt);
\draw [fill=black] (9.5,1.5) circle (0.5pt);
\draw [fill=black] (7.5,0.) circle (0.5pt);
\draw[color=black] (7.53111049328478,-0.23) node {$u_n$};
\draw [fill=black] (6.5,0.) circle (0.5pt);
\draw[color=black] (6.525957480217563,-0.23) node {$v_n$};
\draw [fill=black] (5.5,0.) circle (0.5pt);
\draw[color=black] (5.503014148334996,-0.23) node {$w_n$};
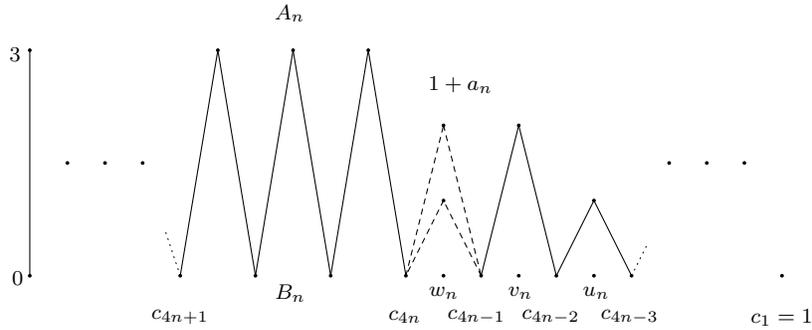
\draw[color=black] (4.5,-1.2) node {Figure 1: The graph of $f_a$};
\end{scriptsize}
\end{tikzpicture}

For every sequence $a\in 2^\N$ we describe a mapping $f_a$ (see Figure 1) which is encoding in some sense the asymptotic behavior of the sequence $a$.
We define $f_a\colon (0,1]\to [0,3]$ in such a way that 
it is continuous, affine on every interval which is disjoint with 
$M:=\bigcup (A_n\cup B_n)\cup \{c_n, u_n, v_n, w_n\colon n\in\N\}$
and $f(c_n)=0$ and $f(u_n)=1$, $f(v_n)=2$ $f(w_n)=1+a_n$, for every $n\in\N$, if $x\in\bigcup B_n$ then $f(x)=0$, if $x\in\bigcup A_n$ then $f(x)=3$.

Let $\psi\colon 2^\N\to\mathcal S([0,1])$ be defined as $\psi(a)=S_{f_a}$.
It can be easily seen that the mapping $\psi\colon 2^\N\to\mathcal K(\R^2)$ is continuous.
It remains to verify that $\psi$ is a reduction.
Clearly if $(a,b)\in E_0$ then $\psi(a)$ is homeomorphic to $\psi(b)$.
On the other hand suppose that $\psi(a)$ is homeomorphic to $\psi(b)$ for some $a,b\in2^\N$. Then by Lemma \ref{homeo} there are homeomorphisms $\alpha\in\mathcal H([0,3])$ and $\beta\in\mathcal H((0,1])$ such that 
$\lim_{t\to0^+}|\alpha(f_a(t))-f_b(\beta(t))|=0$. Note that $\beta$ is clearly increasing.
There exists $\varepsilon\in(0,1)$ such that for $t\in(0,\varepsilon)$ the inequality $|\alpha(f_a(t))-f_b(\beta(t))|<1/4$ holds.

Clearly $\alpha(\{0,3\})=\{0,3\}$ since $\alpha\in\mathcal H([0,3])$. The point $(0,0)$ has the property that whenever there is a sequence of points $x_n\in S_{f_a}$ converging to $x\in \{0\}\times [0,3]$ then there is a sequence $y_n\in S_{f_a}$ converging to $(0,0)$ such that the points $x_n$ and $y_n$ are in the same arc-component and the sequence of minimal arcs containing both $x_n$ and $y_n$ converges to the minimal arc containing $x$ and $(0,0)$.
The point $(0,3)$ in $S_{f_b}$ does not have the property above and hence $\alpha(0)=0$ and $\alpha(3)=3$. 
Thus also $\alpha$ is increasing.
Since moreover the points $(0,i)$ for $i=1,2$ are topologically distinguishable from all the remaing points in $S_{f_a}$ (resp. in $S_{f_b}$) we get that $\alpha(i)=i$ (consider e.g. the topological property that every sufficiently small neighborhood of $(0,1)$ or $(0,2)$ has a sequence of some components whose limit set is not a neighborhood of $(0,1)$ or $(0,2)$ in its arc component).

Let us note that the functions $f_a, f_b$ are piecewise affine on every compact subinterval of $(0,1]$ and thus the mappings $\alpha\circ f_a$ and $f_b\circ\beta$ are piecewise monotone on every such interval. Moreover all the peak points of $\alpha\circ f_a$ and $f_b\circ\beta$ are contained in the set 
$\bigcup_{n\in\mathbb N} A_n\cup\{u_n, v_n, w_n\}$ and
$\beta(\bigcup_{n\in\mathbb N} A_n\cup\{u_n, v_n, w_n\})$
respectively and they are of heights $1,2$ or $3$.

Let $m\in\N$ be such that $c_{m}<\varepsilon$ and consider any $n\geq m$. 
Note that the set $A_n$ (resp. $\beta(A_n)$) is a unique maximal set of $n$-many consecutive peak points of height $3$ for $\alpha\circ f_a$ (resp. $f_b\circ\beta$).
Since moreover $|\alpha(f_a(t))-f_b(\beta(t))|<1/4$, $t\in(0,\varepsilon)$, it follows that $\beta([c_{4n+1}, c_{4n}])\subseteq [u_{n+1},w_n]$ for every $n\geq m$.
We conclude that the three peak points $u_{n+1}, v_{n+1}, w_{n+1}$ for $\alpha \circ f_a$ contained in $[c_{4n+4}, c_{4n+1}]$ of heights $1, 2,1+a_n$ correspond with the three peak points $\beta(u_{n+1}), \beta(v_{n+1}), \beta(w_{n+1})$ for $f_b\circ\beta$ contained in $[\beta(c_{4n+4}), \beta(c_{4n+1})]$ of heights $1,2,1+b_n$.
Especially the heights need to differ by at most $1/2$. Since the heights are integral it follows that $1+a_n=1+b_n$, hence $a_n=b_n$ for $n\geq m$.
Thus $(a,b)\in E_0$.
\end{proof}

\begin{corollary}\label{spirpoli}
The collection of all spirals $\mathcal S([0,1])$ is not uniquely Polishable.
\end{corollary}

\begin{proof}
Suppose for contradiction that there is such a space $X$.
Then by Theorem \ref{uniquePolIsSmooth} the equivalence relation $H\restriction T$ is smooth for every Polish space $T\subseteq\mathcal S([0,1])$.
By Proposition \ref{Enula} the equivalence relation $E_0$ is Borel reducible to $H\restriction T$ for some Polish space $T\subseteq\mathcal S([0,1])$. By transitivity of Borel reducibility this implies that the relation $E_0$ is smooth.
This is a contradiction with the well known fact that the equivalence relation $E_0$ is not smooth \cite[Proposition 6.1.7]{Gao}.
\end{proof}

We can now easily prove the statement in the title.

\begin{corollary}
The collection of all continua is not uniquely compactifiable i.e. there is no compact metrizable space $X$ such that every continuum is homeomorphic to exactly one component of $X$.
\end{corollary}

\begin{proof}
Suppose for contradiction that the collection $\mathcal C$ of all continua is uniquely compactifiable. Then by Lemma \ref{comppoli} it follows that $\mathcal C$ is uniquely Polishable.
Hence the collection of all spirals $\mathcal S([0,1])$ is uniquely Polishable. This is a contradiction with Corollary \ref{spirpoli}.
\end{proof}

We claim without a detailed proof that by similar reasons as given in the proofs above, the class $\mathcal S(K)$ is not uniquely Polishable, for example if $K$ is a simple closed curve.
Also the class of all \emph{dendrites} (see \cite[Chapter X]{Nadler} for the definition) can not be uniquely Polishable because the homeomorphism relation on dendrites is not smooth \cite{CDM} (see also \cite{KrupskiVejnar} for more details on the complexity of homeomorphism relation on compacta).
On the other hand we do not know the answer to the following question.

\begin{question}
Let $K$ be a non-degenerate continuum. Is it always true that $E_0$ is Borel reducible to $H\restriction T$ for some Polish space $T\subseteq\mathcal S(K)$? How is it in the case when $K$ is the pseudo-arc?
\end{question}

I am grateful to Adam Barto\v{s} and Piotr Minc for their comments to this paper.

\bibliographystyle{alpha}
\bibliography{citace}

\newcommand{\etalchar}[1]{$^{#1}$}
\begin{thebibliography}{BBvM{\etalchar{+}}19}

\bibitem[Awa93]{Awartani}
Marwan~M. Awartani.
\newblock An uncountable collection of mutually incomparable chainable
  continua.
\newblock {\em Proc. Amer. Math. Soc.}, 118(1):239--245, 1993.

\bibitem[BBvM{\etalchar{+}}19]{BartosBMPV}
A.~Barto\v{s}, J.~Bobok, J.~van Mill, P.~Pyrih, and B.~Vejnar.
\newblock Compactifiable classes of compacta.
\newblock {\em Topology Appl.}, 266:106836, 25, 2019.

\bibitem[BMPV16]{BartosMarcinaPyrihVejnar}
A.~Barto\v{s}, R.~Marci\v{n}a, P.~Pyrih, and B.~Vejnar.
\newblock Incomparable compactifications of the ray with {P}eano continuum as
  remainder.
\newblock {\em Topology Appl.}, 208:93--105, 2016.

\bibitem[CDM05]{CDM}
Riccardo Camerlo, Udayan~B. Darji, and Alberto Marcone.
\newblock Classification problems in continuum theory.
\newblock {\em Trans. Amer. Math. Soc.}, 357(11):4301--4328, 2005.

\bibitem[For18]{Foreman}
Matthew Foreman.
\newblock What is a {B}orel reduction?
\newblock {\em Notices Amer. Math. Soc.}, 65(10):1263--1268, 2018.

\bibitem[Gao09]{Gao}
Su~Gao.
\newblock {\em Invariant descriptive set theory}, volume 293 of {\em Pure and
  Applied Mathematics (Boca Raton)}.
\newblock CRC Press, Boca Raton, FL, 2009.

\bibitem[IMS15]{IllanesMincSturm}
Alejandro Illanes, Piotr Minc, and Frank Sturm.
\newblock Extending surjections defined on remainders of metric
  compactifications of {$[0,\infty)$}.
\newblock {\em Houston J. Math.}, 41(4):1325--1340, 2015.

\bibitem[Kec95]{Kechris}
Alexander~S. Kechris.
\newblock {\em Classical descriptive set theory}, volume 156 of {\em Graduate
  Texts in Mathematics}.
\newblock Springer-Verlag, New York, 1995.

\bibitem[KV18]{KrupskiVejnar}
Pawe{\l} {Krupski} and Benjamin {Vejnar}.
\newblock {The complexity of homeomorphism relations on some classes of
  compacta}.
\newblock {\em arXiv e-prints}, page arXiv:1808.08760, Aug 2018.

\bibitem[Min16]{Minc}
Piotr Minc.
\newblock {$2^{\aleph_0}$} ways of approaching a continuum with {$[1,\infty)$}.
\newblock {\em Topology Appl.}, 202:47--54, 2016.

\bibitem[Nad92]{Nadler}
Sam~B. Nadler, Jr.
\newblock {\em Continuum theory}, volume 158 of {\em Monographs and Textbooks
  in Pure and Applied Mathematics}.
\newblock Marcel Dekker, Inc., New York, 1992.
\newblock An introduction.

\bibitem[PV12]{PyrihVejnar}
Pavel Pyrih and Benjamin Vejnar.
\newblock Waraszkiewicz spirals revisited.
\newblock {\em Fund. Math.}, 219(2):97--104, 2012.

\bibitem[War32]{Waraszkiewicz}
Zenon Waraszkiewicz.
\newblock Une famille indénombrable de continus plans dont aucun n’est
  l’image d’un autre.
\newblock {\em Fund. Math.}, 18(1):118--137, 1932.

\end{thebibliography}
\end{document}